\documentclass{article}
\title{\textsc{Robust Optimal Control of Arbitrarily Switched Systems: A Path-Complete Framework}%
\thanks{This project has received funding from the European Research Council (ERC) under the European Union's Horizon 2020 research and innovation programme (grant agreement No.~864017, L2C), from the Horizon Europe programme (grant agreement No.~101177842, Unimaas), and from the ARC (French Community of Belgium) project SIDDARTA.}}

\author{
Léa Ninite$^\star$\thanks{Léa Ninite is a Research Fellow of the Fonds de la Recherche Scientifique -- FNRS.\\ E-mails: \texttt{\{lea.ninite,adrien.banse,guillaume.berger,raphael.jungers\}@uclouvain.be}.},
Adrien Banse$^\star$,
Guillaume O.~Berger$^\star$,
Raphaël M.~Jungers$^\star$\\[4pt]
$^\star$ICTEAM Institute, UCLouvain, Louvain-la-Neuve, Belgium
}
\usepackage[T1]{fontenc}
\usepackage[utf8]{inputenc}

\usepackage{amsmath,amssymb,mathtools}
\mathtoolsset{showonlyrefs}

\usepackage{graphicx}
\usepackage{subfig}
\usepackage{tikz}
\usetikzlibrary{arrows.meta,positioning,calc}
\usepackage{circuitikz}
\usepackage[table]{xcolor}

\usepackage{algorithm}
\usepackage{algpseudocode}

\usepackage{booktabs}

\usepackage{fullpage}
\usepackage{comment}
\usepackage{soul}

\usepackage{hyperref}
\hypersetup{
  colorlinks=true,
  linkcolor=black,
  citecolor=black,
  urlcolor=black
}
\usepackage[capitalize]{cleveref}

\newtheorem{definition}{Definition}
\newtheorem{theorem}{Theorem}
\newtheorem{corollary}{Corollary}
\newtheorem{proposition}{Proposition}

\newtheorem{remark}{Remark}

\newenvironment{proof}[1][Proof]{%
  \noindent\textit{#1. }%
}{\hfill$\square$\par}

\DeclareMathOperator*{\argmax}{arg\,max}
\DeclareMathOperator*{\argmin}{arg\,min}

\newcommand{\Mset}{\langle M\rangle}

\renewcommand{\Re}{\mathbb{R}}
\newcommand{\Ne}{\mathbb{N}}

\newcommand{\calS}{\mathcal{S}}
\newcommand{\calF}{\mathcal{F}}
\newcommand{\calE}{\mathcal{E}}
\newcommand{\calG}{\mathcal{G}}
\newcommand{\calT}{\mathcal{T}}
\newcommand{\calH}{\mathcal{H}}

\newcommand{\nxn}{n\times n}

\begin{document}
\maketitle
\begin{abstract}
This paper addresses the robust control of switched systems under arbitrary switching with performance guarantees. We propose a framework that jointly synthesizes a feedback policy and a certified upper bound on its corresponding infinite-horizon closed-loop value function. The proposed upper bound not only certifies the performance of the synthesized policy, but can also be optimized during controller synthesis.
More precisely, our approach associates functions with the nodes of a path-complete graph and enforces graph-based Bellman inequalities along its edges. Exploiting a newly introduced notion of reachability graph, these functions are combined into both a feedback policy and a certified upper bound on its corresponding closed-loop value function, expressed as a pointwise min--max combination of the graph-indexed functions.
For linear switched systems with quadratic stage costs, the proposed framework admits tractable computational formulations based on semidefinite programming and alternating optimization. Numerical experiments, including a building temperature regulation benchmark, demonstrate the practical usefulness of the proposed approach both for direct feedback control using the synthesized policy and for model predictive control using the certified upper bound as a terminal cost.
\end{abstract}

\section{Introduction}

Dynamical systems operating under uncertainty are increasingly prevalent, with applications in autonomous driving, search-and-rescue robotics, smart grids, and smart buildings \cite{Alamir2022,Pairet2022}.
Ensuring the safe deployment of such systems requires rigorous mathematical tools that can formally certify whether prescribed requirements, such as safety and performance specifications, are satisfied \cite{baier2008principles,Lee2016,Mitra2021CPS}.\\\\
In this paper, we consider arbitrarily switched systems (or switched systems for short).
These are multi-modal dynamical systems of the form 
\[
    x_{k+1} = f_{i_k} (x_k, u_k), \quad i_k \in \{1,\ldots,M\}.
\]
Here the mode dynamics $f_i$ are known for every $i$, but the uncertainty comes from the fact that any mode $i$ can be applied at each step $k$.
These systems arise naturally in many applications, including mechanical systems~\cite{blanchini2012constant}, power systems~\cite{sanchez2019practical}, networked systems~\cite{donkers2011stability}, and biological or
epidemiological models~\cite{hernandez2011optimal}; see also \cite{liberzon2003switching} for a comprehensive introduction to switched systems and their applications.
More generally, they provide paradigmatic examples of uncertain systems where robust control is known to be challenging, even in systems with linear dynamics \cite{jungers2009joint,sun2011stability}.\\
This paper is concerned with the robust optimal control of switched systems. Our goal is to design policies that achieve a low \emph{infinite-horizon value function}, which denotes the worst-case infinite-horizon cost-to-go, i.e., the cost under the worst possible infinite mode sequence (see, e.g., \cite{Nilim2005Robust}, and see Definition \ref{definition:cost_to_go_val_function_policy} for a precise definition). Throughout the paper, we will drop the horizon prefix when clear from the context. To achieve tractability and robustness, we design feedback policies that minimize a certified upper bound on the corresponding closed-loop value function, thereby providing an efficient controller while simultaneously addressing the need for guarantees described above. These bounds are relevant in safety and energy-critical applications, where one seeks guarantees on performance and stability even under adversarial operating conditions~\cite{liberzon2003switching,sun2011stability,blanchini2008set}.
They also find applications beyond certification; for instance, they can serve as terminal costs in robust model predictive control (RMPC, also called Min--Max MPC) for switched systems, enabling finite-horizon policies with infinite-horizon performance guarantees (see, e.g., \cite[Chapter~2]{Rawlings} for an introduction).\\\\
Finding such controllers and bounds is a well-understood problem in simpler settings. For example, in the case of linear systems with quadratic stage costs, the value function is quadratic and can be computed efficiently through the solution of the algebraic Riccati equation~\cite{mehrmann1991autonomous}. Similarly, in finite systems, the value function can be easily computed through the robust Bellman equation~\cite{bertsekas2012dynamic,Nilim2005Robust,suilen2024robust} which reduces to a set of piecewise-linear equations to solve.
These properties make the value function both analytically tractable and computationally efficient to evaluate.
The situation changes drastically for switched systems with continuous state and input spaces.
For instance, in the case of switched linear systems with quadratic stage cost and where the mode is a \emph{controlled} input, it was shown in \cite{zhang2009value} that the optimal \emph{finite-horizon} value function is piecewise quadratic, with a number of pieces that may grow exponentially with the horizon, and that the optimal \emph{infinite-horizon} value function is in general neither quadratic nor smooth.
A similar phenomenon arises under \emph{arbitrary} switching, where the number of possible switching sequences grows exponentially with time.
In \cite{berger2025differentiability}, the author shows that, in this setting, the infinite-horizon value function can be non-differentiable almost everywhere, even for switched systems with linear dynamics and quadratic stage cost. These difficulties motivate the development of tractable methods for designing robust controllers, together with guarantees.
To this end, we leverage the powerful \emph{path-complete framework}~\cite{ahmadi2014joint}, initially introduced in the context of stability analysis of switched systems, and combine it with \emph{robust Bellman inequalities}~\cite{Nilim2005Robust,suilen2024robust}, initially introduced in robust dynamic programming to characterize bounds on the value function.
\paragraph*{Contributions}
We use the path-complete framework to define graph-based relaxations of robust Bellman inequalities, leading to a tractable and guaranteed robust control method for arbitrarily switched systems. Our main contributions are as follows:
\begin{itemize}
    \item We associate functions with the nodes of a \emph{path-complete graph} (see Definition~\ref{def:pc_graph}) and enforce graph-based Bellman inequalities along its edges. By exploiting the newly introduced notion of \emph{reachability graph} (Definition~\ref{def:reachability_graph}), we show how these functions can be combined to construct both an explicit feedback policy and a certified upper bound on its corresponding closed-loop value function, expressed as a suitable pointwise min--max combination of the graph-indexed functions.

    \item We derive tractable computational formulations under suitable assumptions. Complete graphs (a particular class of path-complete graphs, see Definition~\ref{definition:complete_cocomplete}) admit an exact semidefinite programming formulation, while general path-complete graphs can be handled through an alternating optimization procedure.

    \item We demonstrate the practical relevance of the proposed framework on a building temperature regulation benchmark. In particular, we show that the certified upper bound can be used as a terminal cost in robust MPC.
\end{itemize}

The main idea of the proposed framework is illustrated in Figure~\ref{fig:overview}.\\\\
This paper extends our preliminary conference version~\cite{ninite2026pathcomplete}. For autonomous switched systems, the conference paper established certified upper bounds on the value function only for complete and co-complete graphs (see Definition~\ref{definition:complete_cocomplete}), while for controlled switched systems it was restricted to complete graphs and piecewise-linear feedback policies. In contrast, the framework developed here applies to general controlled switched systems of the form \(
x_{k+1}=f_{i_k}(x_k,u_k)\), uses arbitrary path-complete graph structures to derive certified upper bounds, and provides a unified methodology for simultaneously synthesizing a feedback policy and a certified upper bound on the corresponding closed-loop value function. This additional flexibility can lead to tighter value function bounds and improved control policies. Finally, the present paper also extends the computational and experimental scope of the conference paper by developing tractable synthesis methods for general path-complete graphs and by demonstrating a practical application of the proposed certified upper bounds as terminal costs in robust MPC.

\definecolor{myblue}{RGB}{44,90,160}
\begin{figure}[H]
\centering
\footnotesize

\begin{tikzpicture}[
    scale=1,
    transform shape,
    >=Latex,
    state/.style={
        circle,
        fill=gray!40,
        minimum size=6.5mm,
        inner sep=0pt
    },
    box/.style={
        rounded corners=7pt,
        fill=gray!15,
        draw=none,
        inner sep=4pt,
        align=center
    }
]

\node[box, minimum width=6.6cm] (top) {
    {\bfseries Robust Bellman inequality}\\[1ex]
    $\displaystyle
    V(x)\ge c(x,\phi(x))
    + \max_i V\!\bigl(f_i(x,\phi(x))\bigr)
    \quad \forall \, x
    $
};

\node[box, below=3.3cm of top, minimum width=.5cm] (bottom) {
{\bfseries Graph-based Bellman inequalities}\\[1ex]

{\color{myblue}
$\displaystyle
V_\alpha(x)\ge c(x,\pi_\alpha(x))
+V_\alpha\!\bigl(f_1(x,\pi_\alpha(x))\bigr)
\quad \forall \, x
$
}\\[0.4ex]

{\color{teal}
$\displaystyle
V_\alpha(x)\ge c(x,\pi_\alpha(x))
+V_\beta\!\bigl(f_2(x,\pi_\alpha(x))\bigr)
\quad \forall \, x
$
}\\[0.4ex]

{\color{red}
$\displaystyle
V_\beta(x)\ge c(x,\pi_\beta(x))
+V_\beta\!\bigl(f_2(x,\pi_\beta(x))\bigr)
\quad \forall \, x
$
}\\[0.4ex]

{\color{orange}
$\displaystyle
V_\beta(x)\ge c(x,\pi_\beta(x))
+V_\alpha\!\bigl(f_1(x,\pi_\beta(x))\bigr)
\quad \forall \, x
$
}
};

\draw[thick,->]
    ($(bottom.north)+(0.55,0.05)$) -- ($(top.south)+(0.55,-0.05)$);

\draw[thick,->]
    ($(top.south)+(-0.55,-0.05)$) -- ($(bottom.north)+(-0.55,0.05)$);

\node[state] (Va) at ($(bottom.north)+(-3.1cm,1.80cm)$) {$V_\alpha$};
\node[state] (Vb) at ($(Va)+(1.35cm,0)$) {$V_\beta$};

\draw[myblue,->]
    (Va.160) to[out=140,in=220,looseness=8]
    node[left,font=\scriptsize] {$1$}
    (Va.200);

\draw[red,->]
    (Vb.20) to[out=40,in=-40,looseness=8]
    node[right,font=\scriptsize] {$2$}
    (Vb.-20);

\draw[teal,->,bend left=25]
    (Va) to node[above,font=\scriptsize] {$2$} (Vb);

\draw[orange,->,bend left=25]
    (Vb) to node[below,font=\scriptsize] {$1$} (Va);

\node[
    draw,
    rounded corners=5pt,
    fill=gray!5,
    align=left,
    inner sep=3pt
] (defs)
at ($(bottom.north)+(2.3cm,1.60cm)$)
{
$\displaystyle V(x)=\min_i V_i(x)$\\
$\displaystyle \kappa(x)=\argmin_i V_i(x)$\\
$\displaystyle \phi(x)=\pi_{\kappa(x)}(x)$
};

\end{tikzpicture}
\caption{Overview of the proposed approach. A path-complete graph is used to generate a collection of graph-based Bellman inequalities, one for each edge of the graph (matching colors indicate the correspondence). Instead of enforcing a single Bellman inequality directly, we enforce these graph-based inequalities on several functions. These functions are then combined to define both a feedback control policy and a certified upper bound on the infinite-horizon value function, which satisfies the robust Bellman inequality.}
\label{fig:overview}
\end{figure}
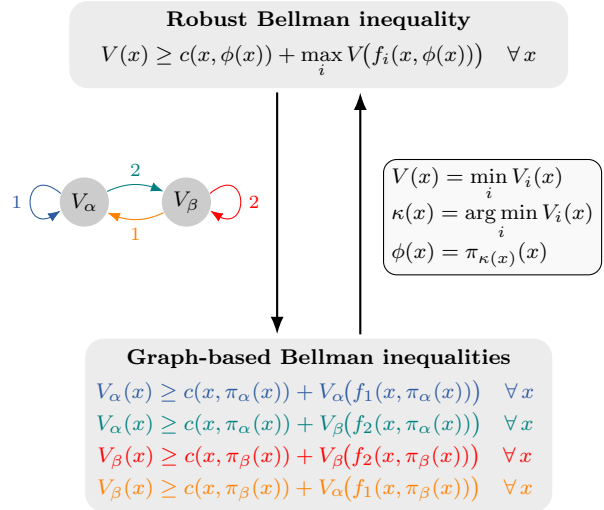

\paragraph*{Related work}
Several works studied the problem of optimal control of switched systems when both the control input and the switching sequence are decision variables (called \emph{optimal control and mode scheduling} problem)~\cite{rantzer2005approximate,gorges2011optimal,zhang2012infinite,wu2020optimal}. In contrast, this paper considers arbitrary switching, which makes the problem fundamentally different; for instance, while optimal control and mode scheduling is a minimization problem over control input and mode sequence, optimal control under arbitrary switching amounts to solve a \emph{minimax} optimization problem (see Definition~\ref{def:optimal-value}).\\\\
Our problem is an instance of robust control.
We present hereafter relevant approaches for robust control and explain the differences with our work.
Robust Lyapunov functions \cite{DEOLIVEIRA1999anew,BLANCHINI1995nonquadratic} and robust barrier certificates \cite{prajna2004safety,prajna2007aframework} focus on stability and safety analysis respectively, leaving the optimal control aspects aside.
Symbolic control~\cite{Belta2017FormalMethods,Tabuada2009} can provide certified controllers for robust optimal control problems~\cite{calbert2021alternating} but this approach, which is very general, usually does not scale with system dimension, since it proceeds by partitioning the state space into small cells.
In particular, it usually does not exploit the structure of switched linear systems.
Robust model predictive control (RMPC)~\cite{Bemporad1999robustmpc} also addresses optimal control problems but faces two main limitations compared to our approach: (i) it does not provide guarantees unless a suitable terminal cost function is used (which essentially requires to solve an RDP problem; see below), and (ii) as we show in our experiments, this approach is computationally demanding because it requires to solve an optimization problem at each time step. Furthermore, the number of variables and constraints of the optimization problem grows exponentially with the prediction horizon.
Finally, robust dynamic programming (RDP)~\cite{fujita2004nondeterministic,suilen2024robust,Nilim2005Robust} provides a powerful mathematical framework to rigorously study robust optimal control problems.
Nevertheless, due to tractability issues, this framework has been mainly used for systems with finite state and input spaces \cite{iyengar2005RDP}.
To the best of our knowledge, our work is the first one to effectively apply RDP to switched systems with continuous state and input spaces, by combining it with path-complete graphs to balance tractability and accuracy.\\\\
Our work leverages path-complete methods, which were introduced for stability analysis of switched systems~\cite{ahmadi2014joint} and have been recently extended to controller synthesis. In particular,~\cite{dellarossa2024graph,lima2025feedback} develop tractable methods for the synthesis of stabilizing controllers using path-complete graphs. However, these works focus on stabilization and restrict attention to a subset of path-complete graphs, called complete graphs. By contrast, the framework developed in this paper addresses performance certification through value function upper bounds and applies to \emph{any} path-complete graphs.

\paragraph*{Outline}
Section~\ref{sec:pb_statement} introduces the problem formulation and
the value function. Section~\ref{sec:preliminaries} reviews the
necessary background on robust Bellman inequalities and path-complete
graphs. The main theoretical results are presented in
Section~\ref{section:ub_general_results}, where path-complete graphs are
used to construct control policies and certified upper bounds on the corresponding value function. Section~\ref{sec:tractable} presents tractable
procedures for computing such policies and bounds. Finally,
Section~\ref{sec:numerical_exp} illustrates the proposed framework on a
toy example and a temperature regulation problem.
It also compares our methods against finite-horizon RMPC, and investigates the use of our upper bounds as terminal costs for RMPC. 

\paragraph*{Notation}
Given $M\in \mathbb N$, we let $\Mset\coloneqq\{1,\dots,M\}$.
The set of $n\times n$ symmetric positive definite (semidefinite) matrices is denoted by $\Re^\nxn_{\succ0}$ ($\Re^\nxn_{\succeq0}$).
The set of nonnegative-valued functions defined on $\Re^n$ (i.e., functions $f:\Re^n\to\Re_{\geq0}$) is denoted by $\calF^n_{\geq0}$. The set of functions from $\Re^n$ to $\Re^m$ is denoted by $(\Re^m)^{\Re^n}$.

\section{Problem statement}\label{sec:pb_statement}
We consider a discrete-time switched system of the form
\begin{equation}
x_{k+1} = f_{\sigma(k)} (x_k, u_k), \quad k\in\Ne,
\label{eq:switched_general}
\end{equation}
where $x_k\in\Re^n$ is the \emph{state} at time $k$,  $u_k\in\Re^m$ is the \emph{control input} at time $k$, $\sigma(k)\in\Mset$ is the \emph{mode} at time $k$ and for all $i\in\Mset$, $f_i:\Re^n \times \Re^m \to\Re^n$.
The function $\sigma : \Ne \to \Mset$ is called the \emph{switching signal}.

Given $x\in\Re^n$, a policy $\phi:\Re^n \to \Re^m$, and a switching signal $\sigma:\Ne\to\Mset$, we denote by $\xi^\phi(\cdot,x,\sigma)$ the solution of~\eqref{eq:switched_general} with switching signal $\sigma$ and policy $\phi$, i.e., $\xi^\phi(k,x,\sigma)=x_k$ for all $k\in\Ne$ where $x_0=x$ and $x_{k+1}=f_{\sigma(k)}(x_k,\phi(x_k))$ for any $k\in \Ne$.

We also consider a cost function $c:\Re^n \times \Re^m \to\Re_{\geq0}$, mapping state--input pairs to nonnegative cost values.

In this work, the switching signal is treated as an external and \emph{uncontrolled} signal.
Accordingly, system performance is evaluated in a worst-case sense over all switching
signals, which motivates the following definitions.

\begin{definition}[Closed-loop value function]\leavevmode\\
Consider the switched system~\eqref{eq:switched_general} under a fixed policy
$\phi:\Re^n \to \Re^m$, and a cost function $c:\Re^n \times \Re^m \to \Re_{\geq 0}$.
For each switching signal $\sigma:\Ne \to \Mset$, the \emph{cost-to-go under policy $\phi$}
is the map $J_\sigma^\phi:\Re^n \to [0,\infty]$ defined by
\[
J_\sigma^\phi(x) \coloneqq \sum_{k=0}^{\infty} 
c\!\left(\xi^\phi(k,x,\sigma),\, \phi(\xi^\phi(k,x,\sigma))\right),
\]
with the convention that $J_\sigma^\phi(x)=\infty$ if the series diverges.
\\
The \emph{closed-loop value function under policy $\phi$} is defined as the worst-case cost over all
switching signals:
\[
J^\phi(x) \coloneqq \sup_{\sigma:\Ne \to \Mset} J_\sigma^\phi(x).
\]
\label{definition:cost_to_go_val_function_policy}
\end{definition}

\begin{definition}[Optimal value function]\label{def:optimal-value}\leavevmode\\
The \emph{optimal value function} of system~\eqref{eq:switched_general} with cost function $c:\Re^n \times \Re^m \to \Re_{\geq 0}$ is defined as the pointwise infimum of the closed-loop value functions over all
policies:
\[
J^\star (x) \coloneqq \inf_{\phi:\Re^n \to \Re^m} J^\phi(x).
\]
\label{definition:optimal_value_function}
\end{definition}
In this paper, our goal is to \emph{jointly} synthesize a policy \(\phi\) and a certified upper bound on its corresponding value function \(J^\phi\), namely a function \(V:\mathbb{R}^n\to\mathbb{R}_{\ge 0}\) satisfying
\[
V(x)\ge J^\phi(x), \qquad \forall \, x\in\mathbb{R}^n.
\]
Note that since \(J^\star(x)\le J^\phi(x)\) for any policy \(\phi\), such a function \(V\) also provides a certified upper bound on the optimal value function \(J^\star\). However, it provides more than an upper bound on \(J^\star\): it certifies the performance of a specific synthesized policy \(\phi\).\footnote{In fact, our framework allows to minimize different metrics associated with the upper bound $V$ (like the value at some given point $x_0$) in order to find policies $\phi$ for which $J^\phi$ is close to $J^\star$.} To address the synthesis problem, we leverage the path-complete framework and robust Bellman inequalities introduced in the next section.

\section{Preliminaries}\label{sec:preliminaries}
\subsection{Robust Bellman inequalities}\label{sec:robust_bellman}
In this section, we draw inspiration from the theory of robust Markov decision processes to develop a dynamic programming framework for switched systems under arbitrary switching. In robust Markov decision processes, the value function is characterized as the solution of a robust Bellman equation, i.e., a recursive functional equation obtained by replacing the expectation in the classical Bellman equation~\cite{bertsekas2012dynamic} with a worst-case operator~\cite{Nilim2005Robust,suilen2024robust}. Related worst-case dynamic programming formulations also arise for nondeterministic systems~\cite{fujita2004nondeterministic}. We recall below the corresponding Bellman equation for switched systems, which serves as the starting point for the Bellman inequalities introduced later in Proposition~\ref{prop:bellman-upper}.

\begin{proposition}[Robust Bellman equation]\leavevmode\\
Consider the switched system~\eqref{eq:switched_general} under a fixed policy
$\phi:\Re^n \to \Re^m$, and a cost function $c:\Re^n \times \Re^m \to \Re_{\ge 0}$. The closed-loop value function $J^\phi$ satisfies the following \emph{robust Bellman equation}:
\begin{equation}
J^\phi(x)
=
c(x,\phi(x))
+
\max_{i\in\Mset} J^\phi(f_i(x,\phi(x))),
\quad \forall \, x\in\Re^n.
\label{eq:robust_bellman_phi}
\end{equation}
\end{proposition}

In practice, when restricting the search to a prescribed function class
(e.g., quadratic functions), it is generally impossible to satisfy the
robust Bellman equation exactly. We therefore consider a relaxation of
this equation, called \emph{robust Bellman inequality}.\footnote{Bellman inequalities have
also been extensively used in stochastic optimal control to derive
tractable bounds on the value function; see,
e.g.,~\cite{lu2021convex,rantzer2005approximate,lincoln2006relaxing}.}
The following proposition, adapted from~\cite[Proposition~2]{ninite2026pathcomplete}, shows that any function satisfying this functional inequality yields an upper bound on the closed-loop value
function.
\begin{proposition}[Robust Bellman inequality]
Consider the switched system~\eqref{eq:switched_general} under a fixed policy
$\phi:\Re^n \to \Re^m$, and a cost function
$c:\Re^n \times \Re^m \to \Re_{\ge 0}$.
Let $J^\phi$ denote the corresponding closed-loop value function. Let $V:\Re^n\to\Re_{\geq0}$ satisfy the following \emph{robust Bellman inequality}:
\begin{equation}
V(x)
\geq
c(x,\phi(x))
+
\max_{i\in\Mset} V(f_i(x,\phi(x))),
\: \: \: \forall \, x\in\Re^n.
\label{eq:robust_bellman_upper}
\end{equation}
Then,
\[
V(x)\geq J^\phi(x),
\quad \forall \, x\in\Re^n.
\]
\label{prop:bellman-upper}
\end{proposition}
\begin{proof}
Let $x\in\Re^n$ and $\sigma:\Ne\to\Mset$.
For each $k\in\Ne$, denote $x_k=\xi^\phi(k,x,\sigma)$.
Note that by~\eqref{eq:robust_bellman_upper}, it holds that for all $k\in\Ne$, $V(x_k)\geq c(x_k,\phi(x_k))+V(x_{k+1})$.
Hence, for all $H\in\Ne_{>0}$,
\begin{align*}
V(x)&\geq V(x)-V(x_H)= \sum_{k=0}^{H-1} V(x_k)-V(x_{k+1}) \geq \sum_{k=0}^{H-1} c(x_k,\phi(x_k)).
\end{align*}
Taking the limit when $H\to\infty$, we get that $V(x)\geq J^\phi_\sigma(x)$.
Then, taking the supremum over $\sigma$, we obtain that $V(x)\geq J^\phi(x)$, concluding the proof. 
\end{proof}

Proposition~\ref{prop:bellman-upper} will be central to our approach for constructing guaranteed upper bounds on the value function using path-complete graphs (introduced below). In particular, we replace the robust Bellman inequality~\eqref{eq:robust_bellman_upper} by a graph-based set of simpler `Bellman-type' inequalities, yielding approximations of the value function expressed as pointwise min--max combinations of functions. 

\subsection{Path-complete graphs}\label{subsection: path-complete graphs}
The path-complete Lyapunov framework, introduced in~\cite{ahmadi2014joint}, extends classical Lyapunov methods for the stability analysis of discrete-time switched systems by replacing a single Lyapunov function with multiple functions coupled through graph-based inequalities. 
The underlying graph, called a \emph{path-complete graph}, is constructed so that every admissible mode sequence corresponds to a path in the graph. 
We next recall the definitions required in the remainder of the paper.\\\\
We consider a \emph{directed labeled graph} $\calG = (\calS, \calE)$, where $\calS$ is a finite set of nodes and $\calE \subseteq \calS \times \calS\times\Mset$ is the set of edges labeled by elements of $\Mset$.
Each edge $(\alpha, \beta, i)\in\calE$ represents a possible transition from node $\alpha$ to node $\beta$ under mode $i$ of the switched system.

\begin{definition}[Path-complete graph]\leavevmode\\
Let $\calG = (\calS,\calE)$ be a directed labeled graph.
We say that $\calG$ is \emph{path-complete} (for $\Mset$) if for any $\ell\in\Ne_{>0}$ and any sequence $\sigma = (i_1,\ldots,i_\ell)\in\Mset^\ell$, there exists a path 
$\{(\alpha_k,\alpha_{k+1}, i_k)\}_{k=1}^\ell\subseteq \calE$.
\label{def:pc_graph}
\end{definition}

An example of path-complete graph is shown in Figure~\ref{fig:pc-graph}. The graph in Figure~\ref{fig:nonpc-graph} is not path-complete (for example, it cannot generate the sequence `$22$').
\begin{figure}[ht]
\centering

\subfloat[]{%
\begin{tikzpicture}[>=Stealth, scale=1, transform shape,
    node distance=2.3cm,
    vertex/.style={circle, draw, minimum size=0.7cm, font=\small,fill=gray!20}]
    
    \node[vertex] (V1) {$\alpha$};
    \node[vertex, right of=V1] (V2) {$\beta$};
    \path[->] (V1) edge[loop above] node{$1$} (V1);
    \path[->] (V2) edge[loop above] node{$2$} (V2);
    \path[->] (V1) edge[bend left=20] node[above]{$1$} (V2);
    \path[->] (V2) edge[bend left=20] node[below]{$2$} (V1);
\end{tikzpicture} \label{fig:pc-graph}
}
\hspace{1.3cm}
\subfloat[]{%
\begin{tikzpicture}[>=Stealth, scale=1, transform shape,
    node distance=2.3cm,
    vertex/.style={circle, draw, minimum size=0.7cm, font=\small, fill=gray!20}]
    
    \node[vertex] (V1) {$\alpha$};
    \node[vertex, right of=V1] (V2) {$\beta$};
    \path[->] (V1) edge[loop above] node{$1$} (U1);
    \path[->] (V1) edge[bend left=20] node[above]{$1$} (V2);
    \path[->] (V2) edge[bend left=20] node[below]{$2$} (V1);
\end{tikzpicture}  \label{fig:nonpc-graph}}

\caption{(a) Path-complete graph with two nodes, for a system with two switching modes. (b) Graph not path-complete.}
\label{fig:pc-vs-copc}
\end{figure}
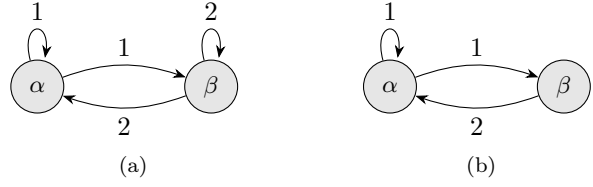

Below, we define two special classes of path-complete graphs:

\begin{definition}[Complete and co-complete graphs]
A directed labeled graph $\calG = (\calS, \calE)$ is \emph{complete} (for~$\Mset$) if for each node $\alpha\in\calS$ and each label $i\in\Mset$, there exists at least one node $\beta\in\calS$ such that $(\alpha, \beta, i)\in\calE$. Conversely, $\calG$ is \emph{co-complete} if for each node $\beta\in\calS$ and each label $i\in\left <M \right >$, there exists at least one node $\alpha\in\calS$ such that $(\alpha, \beta, i)\in\calE$.\label{definition:complete_cocomplete}
\end{definition}
The graph in Figure~\ref{fig:pc-graph} is co-complete, whereas the graph in Figure~\ref{fig:pathcomplete_graph} is path-complete but neither complete nor co-complete.
\begin{figure}[ht]
\centering

\begin{tikzpicture}[>=Stealth, scale=1, transform shape,
    node distance=2.3cm,
    vertex/.style={circle, draw, minimum size=0.7cm, font=\small,fill=gray!20}]
    
    \node[vertex] (V1) {$\alpha$};
    \node[vertex, right of=V1] (V2) {$\beta$};
    \node[vertex, right of=V2] (V3) {$\gamma$};
    \node[vertex, right of=V3] (V4) {$\delta$};
    \path[->] (V4) edge[loop right] node{$2$} (V4);
    \path[->] (V1) edge[bend left=20] node[above]{$1$} (V2);
    \path[->] (V2) edge[bend left=20] node[below]{$1$} (V1);
    \path[->] (V2) edge[bend left=0] node[below]{$1$} (V3);
    \path[->] (V3) edge[bend left=20] node[above]{$1$} (V4);
      \path[->] (V4) edge[bend left=20] node[below]{$2$} (V3);
    \path[->] (V4) edge[bend left=40] node[below]{$2$} (V1);
    \path[->] (V2) edge[bend left=40] node[above]{$1$} (V4);
\end{tikzpicture} 
\caption{Path-complete graph with four nodes, for a system with two switching modes.}
\label{fig:pathcomplete_graph}
\end{figure}
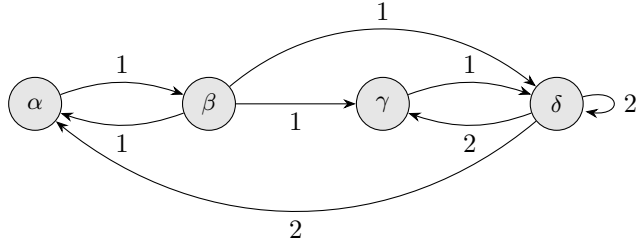
We next introduce the notion of a reachability graph, which will play a central role in the construction of policies and upper bounds on the value function in Section~\ref{section:ub_general_results}.
\begin{definition}[Reachability graph]
\label{def:reachability_graph}
Let $\calG=(\calS,\calE)$ be a path-complete graph for $\Mset$. 
A graph $\calH=(\calS_\calH,\calE_\calH)$, where $\calS_\calH\subseteq2^\calS$ and $\calE_\calH\subseteq\calS_\calH\times\calS_\calH\times\Mset$,\footnote{This means that for all $A\in\calS_\calH$, $A\subseteq\calS$.} is a \emph{reachability graph} of $\calG$ if (i) $\calH$ is complete, and (ii) for every $(A,B,i) \in \calE_\calH$,
\[
\forall \, \beta \in B,\ \exists \, \alpha \in A
\text{ such that }
(\alpha,\beta,i)\in \calE.
\]\vskip0pt%
\end{definition}
\begin{remark}
We note that any graph (with label set $\Mset$) that admits a reachability graph is path-complete (the proof, which is not difficult, is omitted).
Therefore, Definition~\ref{def:reachability_graph} restricts to path-complete graphs.
\end{remark}
Examples of reachability graphs are shown in Figure~\ref{fig:reachability_graph_examples}.
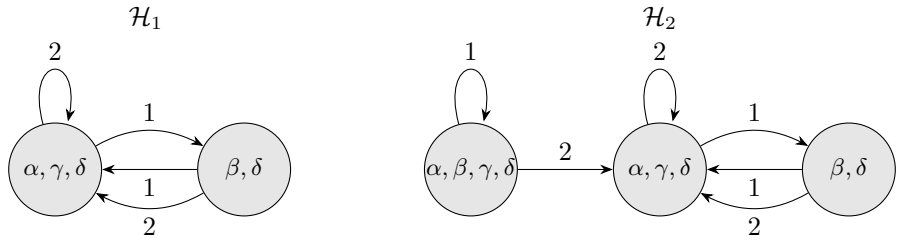
\begin{figure}[ht]
\centering
\begin{tikzpicture}[
>=Stealth,
vertex/.style={
circle,
draw,
minimum size=1.2cm,
text width=1.2cm,
align=center,
inner sep=0pt,
font=\small,
fill=gray!20
}
]

\begin{scope}[xshift=0cm]

\node at (4.2,2) {$\mathcal H_1$};

\node[vertex] (B) at (3,0) {$\alpha,\gamma,\delta$};
\node[vertex] (C) at (5.5,0) {$\beta,\delta$};

\path[->] (B) edge[loop above] node {$2$} ();
\path[->] (B) edge[bend left=30] node[above] {$1$} (C);
\path[->] (C) edge[bend left=30] node[below] {$2$} (B);
\path[->] (C) edge node[below] {$1$} (B);

\end{scope}

\begin{scope}[xshift=8.5cm]

\node at (2.5,2) {$\mathcal H_2$};

\node[vertex] (V1) at (0,0) {$\alpha,\beta,\gamma,\delta$};
\node[vertex] (V2) at (2.5,0) {$\alpha,\gamma,\delta$};
\node[vertex] (V3) at (5,0) {$\beta,\delta$};

\path[->] (V1) edge[loop above] node{$1$} ();
\path[->] (V2) edge[loop above] node{$2$} ();

\path[->] (V1) edge node[above]{$2$} (V2);
\path[->] (V3) edge node[below]{$1$} (V2);

\path[->] (V2) edge[bend left=30] node[above]{$1$} (V3);
\path[->] (V3) edge[bend left=30] node[below]{$2$} (V2);

\end{scope}

\end{tikzpicture}

\caption{Two examples of reachability graphs for the graph in Figure~\ref{fig:pathcomplete_graph}.}
\label{fig:reachability_graph_examples}
\end{figure}
\begin{remark}\label{rem:observer}
Reachability graphs can be viewed as a generalization of the subgraph $O^\star(\calG)$ introduced in~\cite{angeli2017}. Indeed, if $O^\star(\calG)$ denotes the unique strongly connected, deterministic, and complete subgraph of the observer graph of a path-complete graph $\calG$~\cite[Lemma~1]{angeli2017}, then $O^\star(\calG)$ is a reachability graph in the sense of Definition~\ref{def:reachability_graph}. For example, the graph $\calH_1$ in Figure~\ref{fig:reachability_graph_examples} corresponds to $O^\star(\calG)$ with $\calG$ the graph in Figure~\ref{fig:pathcomplete_graph}, and the graph $\calH_2$ is the observer graph. 
\end{remark}

\section{Deriving path-complete control policies and upper bounds on the value function}\label{section:ub_general_results}
In this section, we show how path-complete graphs can be used to jointly construct a feedback control policy and a certified upper bound on the value function of the corresponding closed-loop system. To this end, we introduce two families of decision variables: functions $\{V_\alpha\}_{\alpha\in\calS}$ associated with the nodes of a path-complete graph, and functions $\{\pi_A\}_{A\in\calS_\calH}$ associated with the nodes of one of its reachability graphs (see Definition~\ref{def:reachability_graph}). By enforcing graph-based Bellman inequalities on these functions, we show that they can be combined into an explicit feedback law $\phi$ and a certified upper bound $V$ on $J^\phi$.\\\\
More precisely, given a path-complete graph $\calG=(\calS,\calE)$, let $\calH=(\calS_\calH,\calE_\calH)$ be a reachability graph, and let $\calT\subseteq\calF^n_{\ge0}$ be a template of functions. For system~\eqref{eq:switched_general} and a cost function
$c:\Re^n\times\Re^m\to\Re_{\ge0}$, we seek functions
$\{V_\alpha\}_{\alpha\in\calS}\subseteq\calT$
and $\{\pi_A\}_{A\in\calS_\calH}\subseteq(\Re^m)^{\Re^n}$
satisfying the following graph-based Bellman inequalities:
\begin{align}
&\hspace{-0.4cm}V_\alpha(x)
\ge
c(x,\pi_A(x))
+
V_\beta(f_i(x,\pi_A(x))), \quad \forall\,x\in\Re^n,\;
\forall\,A\in\calS_\calH,\;
\forall\,(\alpha,\beta,i)\in\calE
\text{ with }\alpha\in A.
\label{eq:inequalities_general_case}
\end{align}

Given a solution to~\eqref{eq:inequalities_general_case}, we define the feedback policy
\begin{equation}
\phi(x)\coloneqq \pi_{A^\star(x)}(x), \quad
A^\star(x)\in
\argmin_{A\in\calS_\calH}
\max_{\alpha\in A}V_\alpha(x).
\label{eq:def_policy}
\end{equation}

The theorem below shows that the graph-based Bellman inequalities~\eqref{eq:inequalities_general_case} yield an explicit certified upper bound on the closed-loop value function of the policy~\eqref{eq:def_policy}. The result builds on ideas from~\cite[Theorem~1]{angeli2017} and~\cite[Proposition~10]{dellarossa2024graph}. The former shows how a set of inequalities similar to \eqref{eq:inequalities_general_case} (but without the controllers and the cost functions) provides a common Lyapunov function for autonomous switched systems, by using the observer graph (see Remark~\ref{rem:observer}) as a reachability graph.
The second extends this approach to switched systems with inputs, but restricts to complete graphs and less rich controller structures than in \eqref{eq:def_policy}.
Our theorem is the first to simultaneously (i) provide certified upper bounds on the value function (rather than stability guarantees), (ii) apply to systems with inputs, and (iii) allow for general path-complete graphs and rich controller structures as in \eqref{eq:def_policy}, enabled by the notion of reachability graph.

\begin{theorem}\label{theorem:upper-bound-general}
Consider system~\eqref{eq:switched_general}, a cost function $c:\Re^n\times\Re^m \to\Re_{\geq0}$, a directed labeled graph $\calG=(\calS,\calE)$ and a template $\calT\subseteq\calF^n_{\geq0}$. Assume that $\calG$ is path-complete for~$\Mset$ and let $\calH=(\calS_\calH, \calE_\calH)$ be a reachability graph of $\calG$. Assume furthermore that $\{V_\alpha\}_{\alpha\in\calS}\subseteq\allowbreak\calT$, $\{\pi_A\}_{A \in \calS_\calH} \subseteq (\Re^m)^{\Re^n}$ satisfy~\eqref{eq:inequalities_general_case}. Let $V:\Re^n\to\Re_{\geq0}$ be defined by
\begin{equation}
V(x) \coloneqq \min_{A\in \calS_\calH} \max_{\alpha\in A} V_\alpha(x),
\label{eq:def_value_function_general}
\end{equation}
and let $\phi:\Re^n\to\Re^m$ be defined as in~\eqref{eq:def_policy}. 
Then, for all $x\in \Re^n$, $V(x)\geq J^\phi(x)\geq J^\star (x)$, where $J^\star $ is the optimal value function of system~\eqref{eq:switched_general} with cost $c$, and $J^\phi$ is the value function of the closed-loop system with policy $\phi$.
\end{theorem}

\begin{proof}
    Take any $x\in \Re^n$ and let $A=A^\star(x)$, with $A^\star(x)$ defined in~\eqref{eq:def_policy}. Then, by~\eqref{eq:def_policy} and~\eqref{eq:def_value_function_general},
     \begin{equation}
     \phi(x)=\pi_A(x),
     \qquad
     V(x)=\max_{\alpha \in A} V_\alpha(x). \label{eq:proof_thm}
     \end{equation}
    We show that~\eqref{eq:robust_bellman_upper} is satisfied at $x$. To this end, let $i\in \Mset$. Take any $B\in \calS_\calH$ such that $(A,B,i)\in \calE_\calH$ (such a node $B$ always exists since the reachability graph $\calH$ is complete). By Definition~\ref{def:reachability_graph}, we know that for every $\beta \in B$, there exists $\alpha \in A$ such that $(\alpha, \beta,i)\in \calE$. In particular, for
    \[
    \beta^\star \in \argmax_{\beta \in B} V_\beta(f_i(x,\pi_A(x))),
    \]
    there exists $\alpha^\star \in A$ such that $(\alpha^\star,\beta^\star,i)\in \calE$. Applying~\eqref{eq:inequalities_general_case} to $x$, $A$, and $(\alpha^\star,\beta^\star,i)$, we get
    \[
    V_{\alpha^\star}(x)
    \ge
    c(x,\pi_A(x))
    +
    \max_{\beta \in B} V_\beta(f_i(x,\pi_A(x))).
    \]
    Recalling~\eqref{eq:proof_thm}, we have
    \[
    V(x)\geq c(x,\phi(x)) + V(f_i(x,\phi(x))).
    \]
    Since this holds for every $i\in\Mset$, taking the maximum over $i$ yields
    \[
    V(x)\geq c(x,\phi(x)) + \max_{i\in\Mset}  V(f_i(x,\phi(x))).
    \]
    Since $x$ was arbitrary, we can use Proposition~\ref{prop:bellman-upper} to conclude that $V(x)\ge J^\phi(x)$ for all $x\in \Re^n$. The inequality $J^\phi(x)\ge J^\star(x)$ for all $x\in \Re^n$ follows directly from Definition~\ref{definition:optimal_value_function}, completing the proof.
\end{proof}

The next two corollaries provide specializations of Theorem~\ref{theorem:upper-bound-general}
for complete and co-complete graphs (which are particular classes of path-complete graphs; see Definition~\ref{definition:complete_cocomplete}). In these specializations, the
upper bound simplifies to a pointwise minimum or maximum of the
functions $V_\alpha$, without explicitly involving a reachability graph. These corollaries also generalize~\cite[Theorem~1]{ninite2026pathcomplete} to systems with inputs.

\begin{corollary}
\label{corollary:complete_case}
Consider system~\eqref{eq:switched_general}, a cost function
$c:\Re^n\times\Re^m\to\Re_{\ge0}$, a \emph{complete} graph
$\calG=(\calS,\calE)$, and a template
$\calT\subseteq\calF^n_{\ge0}$. Assume that $\{V_\alpha\}_{\alpha\in\calS}\subseteq\calT$ and
$\{\pi_{\alpha}\}_{\alpha\in\calS}
\subseteq(\Re^m)^{\Re^n}$ satisfy
\begin{equation}
V_\alpha(x)
\ge
c(x,\pi_{\alpha}(x))
+
V_\beta(
f_i(x,\pi_{\alpha}(x))
),\label{eq:corollary_complete}
\end{equation}
for all $x\in\Re^n$ and all $(\alpha,\beta,i)\in\calE$.
Let $V:\Re^n\to\Re_{\geq0}$ be defined by
\begin{equation}
V(x)\coloneqq \min_{\alpha\in\calS}V_\alpha(x), \label{eq:ub_complete}
\end{equation}
and let $\phi:\Re^n\to\Re^m$ be defined by
\begin{equation}
\phi(x)\coloneqq \pi_{\kappa(x)}(x), \qquad \kappa(x)\in\argmin_{\alpha\in\calS}V_\alpha(x). \label{eq:policy_complete}
\end{equation}
Then, for all $x\in\Re^n$, $V(x)\ge J^\phi(x)\ge J^\star(x)$.
\end{corollary}

\begin{proof}
Let $\{V_\alpha\}_{\alpha\in\calS}$ and
$\{\pi_\alpha\}_{\alpha\in\calS}$ satisfy~\eqref{eq:corollary_complete}.

Since $\calG$ is complete, the graph $\calH=(\calS_{\calH},\calE_{\calH})$ with nodes $\calS_{\calH}\coloneqq\{\{\alpha\}:\alpha\in\calS\}$ and edges $\calE_{\calH}$ defined by
\[
(\{\alpha\},\{\beta\},i)\in\calE_{\calH}
\iff
(\alpha,\beta,i)\in\calE,
\]
is a reachability graph of $\calG$.
For each $A=\{\alpha\}\in\calS_{\calH}$, define
\[
\tilde\pi_A\coloneqq\pi_\alpha.
\]
Then, by~\eqref{eq:corollary_complete},  $\{V_\alpha\}_{\alpha\in\calS}$, $\{\tilde \pi_A\}_{A \in \calS_\calH}$ satisfy~\eqref{eq:inequalities_general_case}. Hence all assumptions of
Theorem~\ref{theorem:upper-bound-general} are satisfied.

Since every node of $\calH$ is a singleton, the function defined in
\eqref{eq:def_value_function_general} reduces to~\eqref{eq:ub_complete} and the policy~\eqref{eq:def_policy} reduces to~\eqref{eq:policy_complete}.
Applying Theorem~\ref{theorem:upper-bound-general} yields the result.
\end{proof}

\begin{corollary}
\label{corollary:cocomplete_case}
Consider system~\eqref{eq:switched_general}, a cost function
$c:\Re^n\times\Re^m\to\Re_{\ge0}$, a \emph{co-complete} graph
$\calG=(\calS,\calE)$, and a template
$\calT\subseteq\calF^n_{\ge0}$. Assume that
$\{V_\alpha\}_{\alpha\in\calS}\subseteq\calT$ and
$\phi:\Re^n \to \Re^m$ satisfy
\begin{equation}
V_\alpha(x)
\ge
c(x,\phi(x))
+
V_\beta\!\left(
f_i(x,\phi(x))
\right), \label{eq:ineq_corollary_cocomplete}
\end{equation}
for all $x\in\Re^n$ and all $(\alpha,\beta,i)\in\calE$.
Let $V:\Re^n\to\Re_{\geq0}$ be defined by
\begin{equation}
V(x)\coloneqq \max_{\alpha\in\calS}V_\alpha(x).\label{eq:ub_cocomplete}
\end{equation}
Then, for all $x\in\Re^n$, $V(x)\ge J^\phi(x)\ge J^\star(x)$.
\end{corollary}

\begin{proof}
Let $\{V_\alpha\}_{\alpha\in\calS}$ and $\phi$ satisfy~\eqref{eq:ineq_corollary_cocomplete}.

Since $\calG$ is co-complete, the graph $\calH=(\calS_{\calH},\calE_{\calH})$ with $\calS_\calH=\{\calS\}$ and $\calE_\calH=\calS_\calH\times\calS_\calH\times\Mset$ is a reachability graph of $\calG$.
Associate with its unique node the policy
\[
\pi_{\calS}\coloneqq\phi.
\]
Then, by~\eqref{eq:ineq_corollary_cocomplete},  $\{V_\alpha\}_{\alpha\in\calS}$, $\{\pi_A\}_{A \in \calS_\calH}$ satisfy~\eqref{eq:inequalities_general_case}.
Hence all assumptions of
Theorem~\ref{theorem:upper-bound-general} are satisfied.

Since $\calS_\calH=\{\calS\}$, the function defined in
\eqref{eq:def_value_function_general} reduces to~\eqref{eq:ub_cocomplete} and the policy~\eqref{eq:def_policy} reduces to $\phi$.
Applying Theorem~\ref{theorem:upper-bound-general} yields the result.
\end{proof}

\section{Matrix inequalities in the LQR setting}\label{sec:tractable}

The conditions in \eqref{eq:inequalities_general_case} remain difficult to solve in general because they are \emph{functional} inequalities which implies that they need to be verified or enforced at \emph{every} state $x$.
Several approaches allow to do that (possibly conservatively), including sum-of-square programming~\cite{Parrilo2003} (for inequalities involving polynomials), counterexample-guided methods \cite{chang2019advances}, or sensitivity analysis \cite{Baier2019numerical}.
In this section, we focus on switched systems with linear dynamics, quadratic cost functions, and piecewise-linear policies, for which we can express unconservatively the functional inequalities as \emph{matrix inequalities}.\\\\
More precisely, we focus on switched systems of the form:
\begin{equation}
x_{k+1}=A_{\sigma(k)}x_k+B_{\sigma(k)}u_k, \quad k\in\Ne,
\label{eq:switched_Ax_Bu_controlled}
\end{equation}
wherein $A_i\in\Re^{n\times n}$ and $B_i\in\Re^{n\times m}$ for every $i\in\Mset$. Moreover, we seek local policies $\pi_A$ parameterized as \emph{linear} state-feedback laws:
\begin{equation}
\pi_A(x)=K_Ax,\qquad
K_A\in\Re^{m\times n},
\quad
\forall\,A\in\calS_\calH.
\label{eq:pwl_policy_general}
\end{equation}
The resulting policy~\eqref{eq:def_policy} is therefore piecewise linear.
We further consider a quadratic cost function
\begin{equation}
c(x,u)=x^\top Qx+u^\top Ru,
\label{eq:quadratic_cost_controlled}
\end{equation}
where $Q\in\Re^{n\times n}_{\succ0}$ and $R\in\Re^{m\times m}_{\succ0}$.
Finally, we seek local upper bounds $\{V_\alpha\}_{\alpha\in\calS}$ in the template of quadratic functions; namely, parameterized by positive semidefinite matrices
$\{P_\alpha\}_{\alpha\in\calS}$:
\begin{equation}
V_\alpha(x)=x^\top P_\alpha x, \quad
P_\alpha \in \Re^{n\times n}_{\succeq 0},
\quad \forall \, \alpha\in\calS.
\label{eq:quadratic_template_controlled}
\end{equation}
In the rest of this section, the above structure is assumed.
We explain how the functional inequalities \eqref{eq:inequalities_general_case} can be formulated as \emph{polynomial matrix inequalities} (PMIs) for general path-complete graphs and solved using alternating optimization.
We also show how they can be formulated as \emph{linear matrix inequalities} (LMIs) in the specific case of complete graphs.

\subsection{PMI formulation and alternating optimization}\label{sec:alternating_opt}
We describe here an algorithmic procedure for computing a control policy
together with a certified upper bound on the corresponding closed-loop
value function using a general path-complete graph~$\calG$.

Under the assumptions of this section,~\eqref{eq:inequalities_general_case} reduces to the matrix
inequalities
\begin{equation}
P_\alpha
\succeq
Q
+
K_A^\top R K_A^{} + (A_i+B_iK_A)^\top
P_\beta
(A_i+B_iK_A), \quad \forall\,A\in\calS_\calH,\;
\forall\,(\alpha,\beta,i)\in\calE
\text{ with }\alpha\in A. \label{eq:BMI}
\end{equation}
Since \eqref{eq:BMI} is multilinear in the decision variables $P_\alpha$ and $K_A$, we resort to an alternating optimization procedure, described below.\\\\
\emph{Initialization.} To initialize the algorithm, we assume that a stabilizing policy is available. More precisely, we assume that a feasible solution $\bigl(\{K_A\}_{A\in\calS_\calH},\{P_\alpha\}_{\alpha\in\calS}\bigr)$ to \eqref{eq:BMI} is given. Practical procedures for constructing such an initial stabilizing policy are discussed in Appendix~\ref{app:stabilizing_policy}.

\emph{Loop.} Starting from this feasible solution, we iteratively optimize over the matrices $\{P_\alpha\}_{\alpha\in\calS}$ and the feedback gains $\{K_A\}_{A\in\calS_\calH}$ in alternation while keeping the other variables fixed. More precisely, we repeatedly alternate between the `P-step' and `K-step' described below.

\emph{P-step.}
Given a policy $\{K_A\}_{A\in\calS_\calH}$, the matrix inequalities \eqref{eq:BMI} become linear in the variables $\{P_\alpha\}_{\alpha\in\calS}$. We compute the matrices $\{P_\alpha\}_{\alpha\in\calS}$ by solving the semidefinite program
\begin{equation}
\label{eq:P_step}
\min_{\{P_\alpha\}_{\alpha\in\mathcal S}}
\;
\sum_{\alpha\in\mathcal S} \mathrm{trace}(P_\alpha) \quad \text{s.t.}\quad
\eqref{eq:BMI}.
\end{equation}
The objective function above reflects our goal of obtaining a small upper
bound on the closed-loop value function. We adopt the trace-minimization
criterion used in~\cite{ninite2026pathcomplete}, although other choices, such as Frobenius norm or log-det minimization,\footnote{Note that the convex reformulation for the log-det minimization requires a Schur complement step to obtain $P_\alpha^{-1}$.} are also possible~\cite{ninite2026pathcomplete}.

\emph{K-step.}
Given the matrices $\{P_\alpha\}_{\alpha\in\calS}$, we solve the following optimization problem: 
\begin{align}
\min_{\gamma \ge 0, \{K_A\}_{A\in\mathcal S_\mathcal H}}\quad & \gamma \\
\text{s.t.}\quad &
\gamma P_\alpha
-
Q
-
K_A^\top R K_A \succeq(A_i+B_iK_A)^\top \gamma P_\beta (A_i+B_iK_A),\!\!\!\! \\
&\hspace{3cm}\forall\,A\in\calS_\calH,\;\forall\,(\alpha,\beta,i)\in\calE
\text{ with }\alpha\in A.
\label{eq:K_step}
\end{align}
It is solved by bisection on $\gamma$.
For any fixed value of $\gamma$, a Schur complement reformulation of the
constraints (provided in Appendix~\ref{sec:SDP_formu}) yields a semidefinite program in the variables
$\{K_A\}_{A\in\mathcal S_\mathcal H}$.

The scalar $\gamma$ is introduced to reflect our overall objective of decreasing the certified upper bound on the value function. By minimizing $\gamma$, we seek the smallest scaling of the matrices $\{P_\alpha\}_{\alpha\in\mathcal S}$ for which the inequalities remain feasible.\\\\
The overall procedure is summarized in Algorithm~\ref{alg:alternating_optimization}.
\begin{algorithm}[ht]
\caption{Alternating optimization}
\label{alg:alternating_optimization}
\begin{algorithmic}[1]
\State Choose a path-complete graph $\calG=(\calS,\calE)$ and a reachability graph $\calH=(\calS_\calH,\calE_\calH)$ of $\calG$
\State Initialize
$\bigl(\{K_A^{(0)}\}_{A\in\mathcal S_\mathcal H},
\{P_\alpha^{(0)}\}_{\alpha\in\mathcal S}\bigr)$
with a feasible solution to \eqref{eq:BMI} (see Appendix~\ref{app:stabilizing_policy})

\State $k\gets 0$

\Repeat

    \State \textbf{P-step}

    \State Compute $\{P_\alpha^{(k+1)}\}_{\alpha\in\mathcal S}$
    by solving \eqref{eq:P_step} with $\{K_A\}_{A\in\mathcal S_\mathcal H}$ fixed at $\{K_A^{(k)}\}_{A\in\mathcal S_\mathcal H}$

    \Statex

    \State \textbf{K-step}

    \State Compute $\{K_A^{(k+1)}\}_{A\in\mathcal S_\mathcal H}$ by solving~\eqref{eq:K_step}
    with $\{P_\alpha\}_{\alpha\in\mathcal S}$ fixed at $\{P_\alpha^{(k+1)}\}_{\alpha\in\mathcal S}$

    \State $k\gets k+1$

\Until{convergence}

\State \Return
$\{K_A^{(k)}\}_{A\in\mathcal S_\mathcal H}$
and
$\{P_\alpha^{(k)}\}_{\alpha\in\mathcal S}$

\end{algorithmic}
\end{algorithm}

\subsection{LMI formulation in the complete case}\label{sec:LMIs}
In this section, we show that for \emph{complete} graphs (see Definition~\ref{definition:complete_cocomplete}), the PMIs~\eqref{eq:BMI} can be reformulated as LMIs. As a consequence, the synthesis of a policy together with a certified upper bound on the corresponding closed-loop value function can be carried out through a single semidefinite program, thereby avoiding the alternating optimization procedure described above.\\\\
As shown in Corollary~\ref{corollary:complete_case}, complete graphs
allow assigning a distinct policy $\pi_\alpha$ to each node $\alpha$ of the graph. In the
linear-feedback setting considered here, this corresponds to assigning a
matrix $K_\alpha\in\Re^{m\times n}$ to each node $\alpha\in\calS$. Under the assumptions of this section, the feedback policy~\eqref{eq:policy_complete} reduces to the piecewise-linear controller
\begin{equation}
\phi(x)=K_{\kappa(x)}x,
\qquad
\kappa(x)\in\argmin_{\alpha\in\calS} x^\top P_\alpha x.
\label{eq:PWL_policy}
\end{equation}
Moreover, inequalities~\eqref{eq:BMI} become
\begin{align}
P_\alpha \succeq Q + K_\alpha^\top R K_\alpha^{} + (A_i + B_i K_\alpha)^\top P_\beta (A_i + B_i K_\alpha), \qquad \forall \, (\alpha,\beta,i)\in\calE.
\label{eq:matrix_inequality_controlled}
\end{align}
These matrix inequalities can be expressed as the following LMIs, as shown in our preliminary work~\cite[Proposition~7]{ninite2026pathcomplete}:
\begin{align}
&\begin{bmatrix}
S_\alpha & S_\alpha A_i^\top + Y_\alpha^\top B_i^\top & S_\alpha & Y_\alpha^\top \\
A_i S_\alpha + B_i Y_\alpha & S_\beta & 0 & 0 \\
S_\alpha & 0 & Q^{-1} &0 \\
Y_\alpha & 0 & 0 & R^{-1}
\end{bmatrix} \succeq0, \qquad \forall\,(\alpha,\beta,i)\in\calE,
\label{eq:lmi_controlled}
\end{align}
where $S_\alpha=P_\alpha^{-1}$ and
$Y_\alpha=K_\alpha P_\alpha^{-1}$.

Consequently, the joint computation of a control policy and an upper bound on the corresponding closed-loop value function reduces to solving the following SDP:\footnote{An alternative objective is to minimize $\sum_{\alpha\in\calS}\mathrm{trace}(P_\alpha)$. This can be reformulated as an SDP by introducing variables $P_\alpha$ satisfying $\begin{bmatrix}
    P_\alpha & I\\
    I & S_\alpha 
\end{bmatrix}\succeq 0$.}
\begin{align}
\max_{\{S_\alpha,Y_\alpha\}_{\alpha\in\calS}}
\quad &
\sum_{\alpha\in\calS}\log \det (S_\alpha) \quad
\text{s.t.}\quad
\eqref{eq:lmi_controlled}.
\label{eq:SDP_complete}
\end{align}
Once a solution is obtained, the matrices
$P_\alpha=S_\alpha^{-1}$ and
$K_\alpha=Y_\alpha S_\alpha^{-1}$ define both the policy
\eqref{eq:PWL_policy} and the certified upper bound of
Corollary~\ref{corollary:complete_case}.
\section{Numerical experiments}\label{sec:numerical_exp}
All experiments are conducted under the assumptions of Section~\ref{sec:tractable} (linear dynamics, linear state-feedback laws $\pi_A(x)=K_Ax$, quadratic stage costs, and quadratic template for the functions $V_\alpha$) and performed in Julia using JuMP and MOSEK on a laptop equipped with an Apple M5 chip and 16GB of RAM. Unless stated otherwise, we take $Q=I$ and $R=I$. The code for reproducing the experiments can be found at
\url{https://github.com/lninite/PathCompleteControl.jl}.\\\\
We consider both complete and non-complete path-complete graphs in order to illustrate the two tractable approaches detailed in Section~\ref{sec:tractable}. To instantiate these two classes of graphs, we use primal and dual De Bruijn graphs~\cite{de1948combinatorial}.
These are two classical families of path-complete graphs that have been extensively employed in the analysis of switched systems and in the construction of path-complete Lyapunov functions; see, e.g., \cite{ahmadi2014joint,dellarossa2024graph,dellarossa2024multiple}. Their definition is recalled in Appendix~\ref{sec:DB}.

Primal De Bruijn graphs are complete, while their duals are co-complete. In the following, controllers and performance certificates associated with primal De Bruijn graphs are computed through the SDP~\eqref{eq:SDP_complete}, whereas those associated with dual De Bruijn graphs are computed using Algorithm~\ref{alg:alternating_optimization}.
\subsection{Synthetic example}
We consider the following system matrices:
\[
A_1=
\begin{bmatrix}
0 & 1\\
-1 & 0
\end{bmatrix},
\:
A_2=
\begin{bmatrix}
-0.1 & 0\\
0 & -0.95
\end{bmatrix},\:
B_1=B_2=
\begin{bmatrix}
1\\
0
\end{bmatrix}.
\]
Figure~\ref{fig:upperbounds_toy_example} illustrates the certified upper bounds obtained from primal and dual De Bruijn graphs of orders $\ell\in\{1,2,3,4\}$, evaluated along the unit circle $x=(\cos\theta,\sin\theta)^\top$.\footnote{Since the certified upper bounds are pointwise min--max combinations of quadratic functions, they are homogeneous of degree two. It is therefore sufficient to evaluate them on the unit circle, in dimension $2$.} For both graph families, increasing the graph order generally leads to smaller certified upper bounds, although this improvement is not uniform over the state space. The figure also highlights the diversity of certificates that can be obtained from different path-complete graph structures: depending on the state, either the primal or the dual De Bruijn graph may provide the smaller certified upper bound. This illustrates the interest of extending our previous work~\cite{ninite2026pathcomplete} (which was restricted to complete graphs) to general path-complete graphs.
\begin{figure}[ht]
\setlength{\fboxsep}{0pt}
\centering
    \includegraphics[trim={2.1cm 0.9cm 0.2cm 0.4cm},clip,width=0.65\columnwidth]{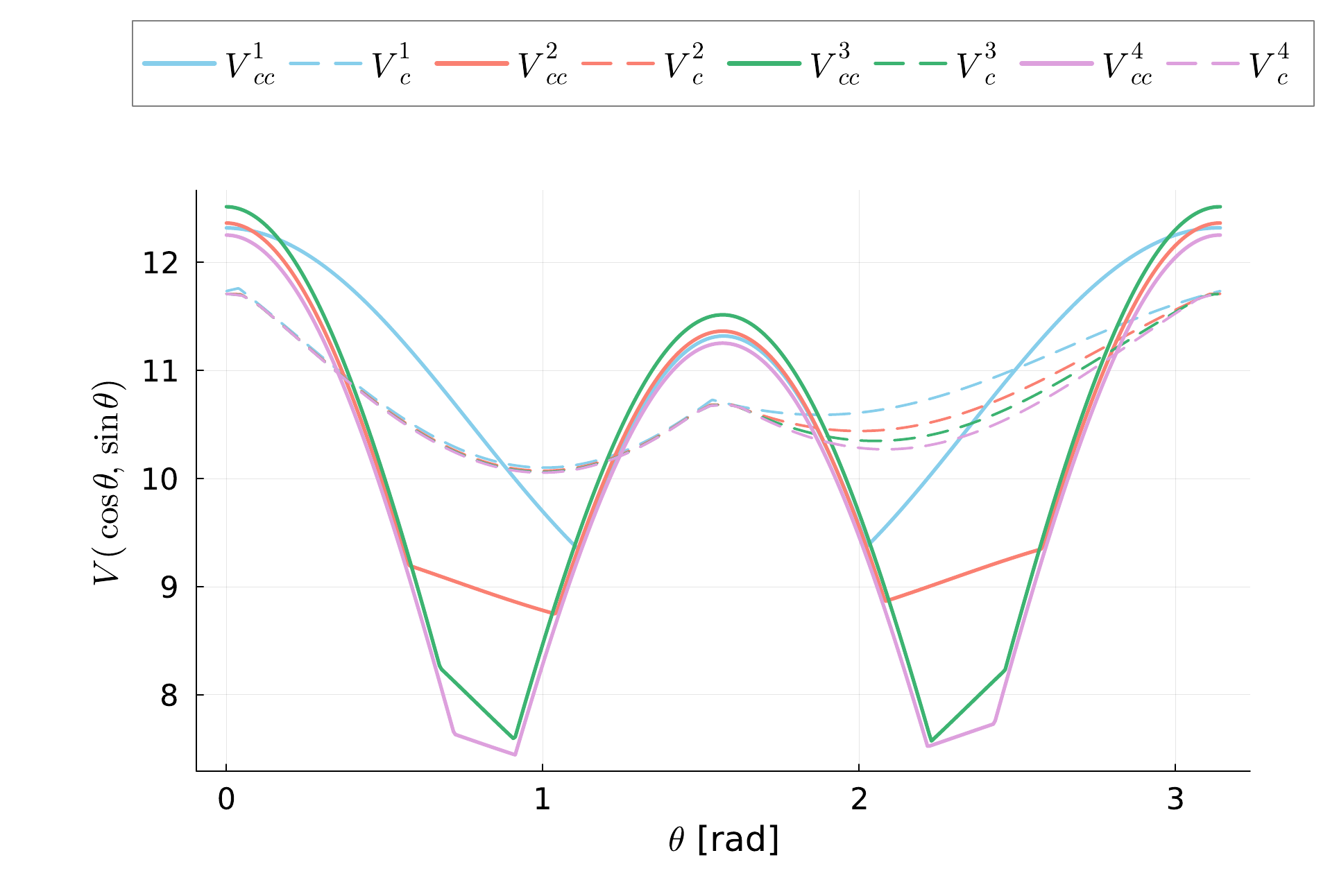}
\caption{Certified upper bounds $V_c^\ell$ and $V_{cc}^\ell$ on the value function, obtained from primal and dual De Bruijn graphs of order $\ell\in\{1,2,3,4\}$, respectively. The bounds $V_c^\ell$ are computed via the SDP~\eqref{eq:SDP_complete}, whereas $V_{cc}^\ell$ are computed using Algorithm~\ref{alg:alternating_optimization}. The bounds are evaluated along the unit circle $x=(\cos\theta,\sin\theta)^\top$.}
    \label{fig:upperbounds_toy_example}
\end{figure}
\subsection{Application to building temperature regulation}

Next, we illustrate the proposed approach on a building temperature regulation benchmark inspired by~\cite[Section VI-C]{wang2024datadriven}.

The system consists of three thermally coupled zones equipped with heating/cooling actuators. A schematic representation is shown in Figure~\ref{fig:building_layout}. The opening and closing of the two internal doors modify the thermal couplings between neighboring zones. Since each door can independently be either open or closed, this results in \(M=4\) switching modes.

The system is modeled as a discrete-time switched linear system of the form~\eqref{eq:switched_Ax_Bu_controlled}, where \(x_k\in\mathbb{R}^3\) denotes the deviation of the zone temperatures from the reference temperature \(T_{\mathrm{ref}}=24^\circ\mathrm{C}\), \(u_k\in\mathbb{R}^3\) contains the control input of each zone, and \(\sigma(k)\in\{1,2,3,4\}\) specifies the door configuration at time \(k\). The system matrices are derived from the RC model of~\cite{wang2024datadriven}, to which we refer for details on the physical modeling and parameter values.
\begin{figure}[H]
\setlength{\fboxsep}{0pt}
\centering
\includegraphics[trim={6mm 4mm 6mm 4mm},clip,width=0.5\linewidth]{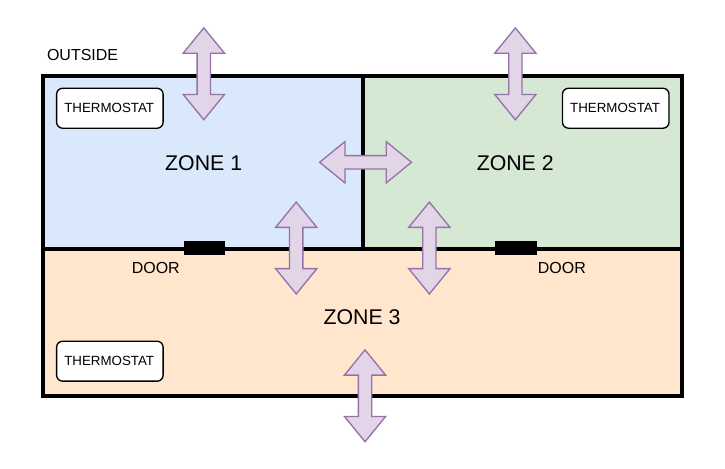}
\caption{Three-zone building temperature regulation benchmark.}
\label{fig:building_layout}
\end{figure}

\paragraph*{Considered controllers} We use this benchmark to compare four controller design approaches: (1) robust MPC without terminal cost; (2) robust MPC with terminal cost; (3) path-complete controllers based on primal De Bruijn graphs; and (4) path-complete controllers based on dual De Bruijn graphs. The objective is to compare the resulting closed-loop performance together with the online and offline computation times of the different approaches. The two robust MPC baselines are briefly described below.

For the robust MPC without terminal cost, a finite-horizon min--max optimal control problem is solved online at each time step in a receding-horizon fashion. Switching uncertainty is represented through a scenario tree containing all admissible mode sequences over the prediction horizon, and the optimization minimizes the worst-case quadratic cost across all scenarios. This leads to a total of $\approx \frac134^N$ decision variables, where $N$ is the prediction horizon.

We also investigate the use of our framework for the design of MPC terminal costs. In MPC, the finite-horizon objective only accounts for costs over the prediction window $\{0,\ldots,N-1\}$ and therefore does not, in general, guarantee \emph{recursive optimal cost decrease}, i.e., that the optimal value of the MPC optimization problem decreases from one time step to the next. If a terminal cost $V_f(x_N)$ is added to the objective and $V_f$ satisfies the robust Bellman inequality~\eqref{eq:robust_bellman_upper}, then recursive optimal cost decrease is guaranteed~\cite[Proposition~2.34]{Rawlings}. Since the value function certificates computed by our framework satisfy the robust Bellman inequality, they can naturally serve as terminal penalties for robust MPC.
In our experiment, we consider two choices: a terminal cost induced by the single-node graph, and a path-complete terminal cost induced by a first-order dual De Bruijn graph. For the single-node graph, the terminal cost is given by
$$
V_f(x_N)\coloneqq x_N^\top P x_N^{},
$$
where $P$ is computed by solving the SDP~\eqref{eq:SDP_complete}. For the dual De Bruijn graph, the terminal cost is given by
$$
V_f(x_N)\coloneqq \max_{\alpha\in\calS} x_N^\top P_\alpha x_N^{},
$$
where the matrices $\{P_\alpha\}_{\alpha\in\calS}$ are computed using Algorithm~\ref{alg:alternating_optimization}. In both cases, $N$ denotes the prediction horizon and the resulting terminal cost provides a certified upper bound on the infinite-horizon cost-to-go.

\paragraph*{Experimental setup} The control objective is to regulate the temperature of the three zones to the constant reference temperature \(T_{\mathrm{ref}} = 24^\circ\mathrm{C}\), which corresponds to driving the state to the origin. The simulations are initialized from zone temperatures \((29,19,29)^\top\,^\circ\mathrm{C}\), corresponding to the initial deviation
\(
x_0 = (5,-5,5)^\top.
\)

For the MPC experiments (with and without terminal costs), we consider prediction horizons $N\in\{2,3,4,5\}$. Larger horizons were not considered because the number of scenarios grows exponentially with $N$, yielding rapidly increasing online computational costs. For the path-complete approach, we consider both primal and dual De Bruijn graphs, with orders \(\ell\in\{1,\dots,4\}\). For primal De Bruijn graphs, the controller and its certified value function upper bound are obtained through the single SDP~\eqref{eq:SDP_complete}. For dual De Bruijn graphs, they are computed using the alternating optimization procedure of Algorithm~\ref{alg:alternating_optimization}.

For each controller, we simulate the closed-loop system, starting from $x_0$, over a horizon \(T=300\). Results are averaged over \(50\) Monte Carlo simulations. In each trial, a switching sequence \(\sigma(0),\ldots,\sigma(T-1)\) is generated by sampling the mode at each time step,  uniformly and independently from \(\{1,2,3,4\}\). The \emph{same set of sampled switching sequences} is used for all controllers to ensure a fair comparison.

\newcommand{\Jmc}{J^{\mathrm{avg}}}

\paragraph*{Performance evaluation} Table~\ref{tab:comparison} reports the average finite-horizon cumulative cost
\[
\Jmc_T=\frac1{50}\sum_{i=1}^{50}J_{i,T},\quad J_{i,T}=\sum_{k=0}^{T-1} x_{i,k}^\top Qx_{i,k}^{}+u_{i,k}^\top Ru_{i,k}^{},
\]
where $(x_{i,\cdot},u_{i,\cdot})$ is the closed-loop trajectory associated with the $i$\textsuperscript{th} Monte Carlo simulation; and this is done separately for each controller.
Since $J_{i,T}$ is evaluated over a finite horizon, it provides a lower bound on the corresponding infinite-horizon cost $J^\phi(x_0)$. For the path-complete controllers, we additionally report the certified upper bound $V(x_0)$. We also report the offline synthesis time $t_{\mathrm{off}}$, the average online execution time $t_{\mathrm{on}}^{\mathrm{avg}}$, and the average total computation time $t_{\mathrm{tot}}^{\mathrm{avg}}$.\\\\
Table~\ref{tab:comparison} highlights several trends. First, our proposed framework achieves computation times that are orders of magnitude lower than those of robust MPC while maintaining comparable closed-loop performance. Comparing the dual De Bruijn controller of order $\ell=1$ with robust MPC equipped with the same terminal cost, both approaches achieve nearly identical average costs ($1196.57$ versus $1196.60$). However, robust MPC requires $1305.11$ seconds of total computation time and $1303.56$ seconds of average online computation time, whereas the proposed controller requires only $1.51$ seconds and $9\times10^{-6}$ seconds, respectively.

Second, terminal costs derived from our framework significantly improve the performance of robust MPC. For instance, at horizon $N=5$, the average cost decreases from $1486.71$ without terminal cost to $1244.74$ when using a common quadratic terminal cost, and further to $1196.60$ when using a dual De Bruijn terminal cost.

Third, the results highlight the benefit of extending the framework beyond complete graphs. While our previous work~\cite{ninite2026pathcomplete} was restricted to complete path-complete graphs, dual (co-complete) De Bruijn graphs achieve lower costs than primal (complete) ones on this benchmark while preserving certified guarantees. This improved performance comes at the expense of higher offline computational costs due to the use of alternating optimization instead of a single semidefinite program.

Finally, increasing the graph order improves both empirical performance and certified upper bounds for primal De Bruijn graphs on this benchmark. In contrast, dual De Bruijn graphs exhibit only minor variations in cost and upper bound as the order increases.\\\\
Figure~\ref{fig:trajectories} illustrates closed-loop trajectories of the three-zone building under four controllers: robust MPC without terminal cost ($N=5$), robust MPC with a first-order dual De Bruijn terminal cost ($N=5$), and path-complete controllers synthesized using primal and dual De Bruijn graphs of order $\ell=4$. The trajectories correspond to a common switching realization and the same initial condition as in the previous experiments.

We observe that the controller synthesized using the primal De Bruijn graph yields a faster convergence to the reference temperature. Moreover, the trajectories obtained with the dual De Bruijn controller and with MPC equipped with a dual De Bruijn terminal cost are nearly indistinguishable, despite the substantial difference in online computational cost highlighted in Table~\ref{tab:comparison}. Finally, the slower convergence of MPC without terminal cost further illustrates the benefit of using guaranteed approaches such as incorporating a terminal cost satisfying the robust Bellman inequality.

\section{Conclusion}
We introduced a path-complete framework for the joint synthesis of control policies and certified upper bounds on their corresponding closed-loop value functions for switched systems subject to arbitrary switching. By associating local functions with the nodes of a path-complete graph and enforcing graph-based Bellman inequalities along its edges, we showed how to construct both a feedback policy and a certified upper bound on its value function through a suitable min--max aggregation of the graph-indexed functions. Previously proposed constructions based on particular classes of path-complete graphs were recovered as special cases of this general result. For linear switched systems with quadratic costs, we further derived tractable computational formulations for controller synthesis and performance certification. Numerical experiments illustrate the proposed approach on several examples, including a building temperature regulation benchmark. The results show that the proposed framework can drastically reduce both total and online computation times compared with robust MPC while maintaining comparable or better closed-loop performance. They further illustrate that the proposed value function certificates are useful beyond controller synthesis, serving as terminal costs for robust MPC.

Future work will investigate the conservatism of the proposed framework, in particular its ability to approximate the optimal value function and optimal policy as the complexity of the underlying path-complete graph increases. Another promising direction is the development of policy-improvement schemes based on path-complete value function certificates, as well as the extension of these ideas beyond switched systems to broader robust decision-making settings such as robust Markov decision processes.
\begin{table}[ht]
\caption{Comparison of robust MPC and path-complete controllers in terms of empirical average cost $\Jmc_T$, offline computation time $t_{\mathrm{off}}$, average online computation time $t_{\mathrm{on}}^{\mathrm{avg}}$, and average total computation time $t_{\mathrm{tot}}^{\mathrm{avg}}=t_{\mathrm{off}}+t_{\mathrm{on}}^{\mathrm{avg}}$. For path-complete controllers, the certified upper bound $V(x_0)$ on the value function is also reported. MPC is evaluated for prediction horizons $N\in\{2,3,4,5\}$, with and without terminal costs, and path-complete controllers are synthesized using primal and dual De Bruijn graphs of orders $\ell\in\{1,2,3,4\}$. Results are averaged over 50 simulations.}
\centering
\setlength{\tabcolsep}{3pt}
\renewcommand{\arraystretch}{0.95}

\begin{tabular}{@{}lccccc@{}}
\toprule
Config. & $\Jmc_T$ & $V(x_0)$ & $t_{\mathrm{tot}}^{\mathrm{avg}}$ [s] & $t_{\mathrm{off}}$ [s] & $t_{\mathrm{on}}^{\mathrm{avg}}$ [s] \\
\midrule

\rowcolor{gray!15}
\multicolumn{6}{c}{\textbf{MPC without terminal cost}} \\
\midrule
$N=2$ & $1562.43$ & -- & 0.78 & -- & 0.78 \\
$N=3$ & $1534.88$ & -- & 2.52 & -- & 2.52 \\
$N=4$ & $1509.72$ & -- & 13.45 & -- & 13.45 \\
$N=5$ & $1486.71$ & -- & 80.36 & -- & 80.36 \\

\midrule
\rowcolor{gray!15}
\multicolumn{6}{c}{\textbf{MPC + common quadratic terminal cost}} \\
\midrule
$N=2$ & $1254.94$ & -- & 2.72 & 0.04 & 2.67 \\
$N=3$ & $1251.29$ & -- & 9.02 & 0.04 & 8.98 \\
$N=4$ & $1247.86$ & -- & 41.06 & 0.04 & 41.02 \\
$N=5$ & $1244.74$ & -- & 400.40 & 0.04 & 400.36 \\

\midrule
\rowcolor{gray!15}
\multicolumn{6}{c}{\textbf{MPC + dual De Bruijn terminal cost ($\boldsymbol{\ell=1}$)}} \\
\midrule
$N=5$ & $1196.60$ & -- & 1305.11 & 1.55 & 1303.56 \\

\midrule
\rowcolor{gray!15}
\multicolumn{6}{c}{\textbf{Path-complete controller (primal De Bruijn graph)}} \\
\midrule
$\ell=1$ & $1261.05$ & 1657.57 & 0.11 & 0.11 & 0.00005 \\
$\ell=2$ & $1261.03$ & 1657.51 & 0.39 & 0.39 & 0.00010 \\
$\ell=3$ & $1226.18$ & 1527.94 & 1.28 & 1.28 & 0.00032 \\
$\ell=4$ & $1214.59$ & 1476.04 & 5.78 & 5.78 & 0.00124 \\

\midrule
\rowcolor{gray!15}
\multicolumn{6}{c}{\textbf{Path-complete controller (dual De Bruijn graph)}} \\
\midrule
$\ell=1$ & $1196.57$ & 1279.75 & 1.51 & 1.51 & 0.000009 \\
$\ell=2$ & $1196.52$ & 1279.74 & 4.16 & 4.16 & 0.000009 \\
$\ell=3$ & $1196.50$ & 1279.72 & 12.66 & 12.66 & 0.000009 \\
$\ell=4$ & $1196.47$ & 1279.71 & 55.06 & 55.06 & 0.000009 \\

\bottomrule
\end{tabular}
\label{tab:comparison}
\end{table}
\begin{figure}[H]
\setlength{\fboxsep}{0pt}
\centering
    \includegraphics[trim={11mm 14mm 12mm 10mm},clip,width=0.55\columnwidth]{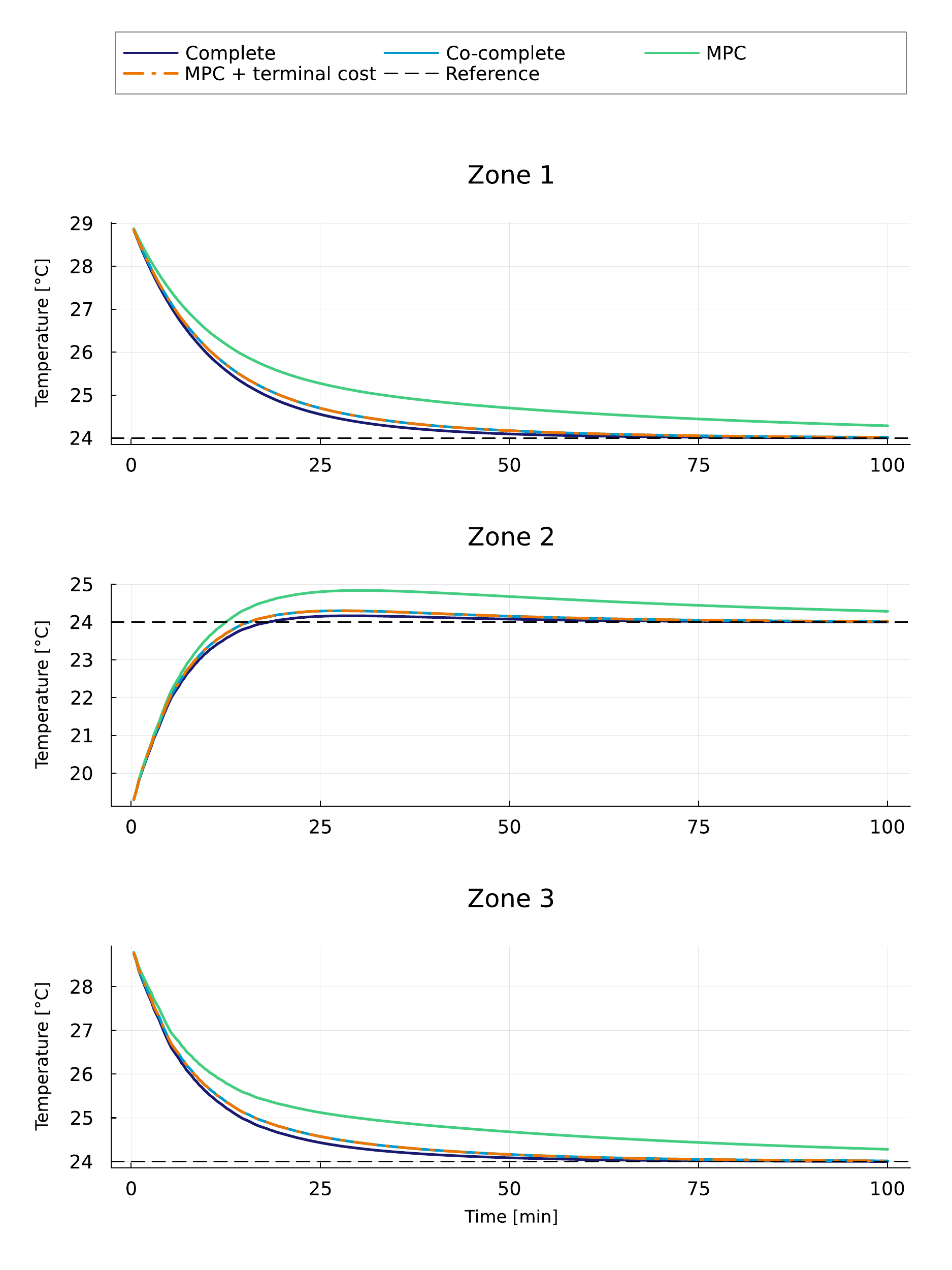}
\caption{Closed-loop state trajectories under four controllers: the proposed path-complete controller based on the primal De Bruijn graph of order 4 (complete), the controller based on the dual De Bruijn graph of order 4 (co-complete), robust MPC with horizon $N=5$, and robust MPC with horizon $N=5$ using a terminal cost derived from the dual De Bruijn graph of order 1. All controllers are simulated under the same switching sequence.}

    \label{fig:trajectories}
\end{figure}
\clearpage
\bibliographystyle{abbrv}
\bibliography{refs}
\appendix
\section{Initialization of Algorithm~\ref{alg:alternating_optimization}}\label{app:stabilizing_policy}
The alternating optimization procedure described in Algorithm~\ref{alg:alternating_optimization} requires an initial feasible point (line 2). To obtain one, we first solve a stabilization problem. To this end, we introduce a scaling factor $\gamma\geq0$ and seek a collection of matrices $\{P_\alpha\}_{\alpha\in\calS}$ and controller gains $\{K_A\}_{A\in\calS_\calH}$ such that
\begin{equation}
\gamma P_\alpha
-(A_i+B_iK_A)^\top P_\beta (A_i+B_iK_A)
\succeq 0,
\label{eq:stabilization_LMIs}
\end{equation}
for every transition $(\alpha,\beta,i)\in\calE$ with $\alpha\in A$. Any feasible solution satisfying $\gamma<1$ guarantees the existence of a stabilizing controller.

Starting from an initial collection of positive definite matrices $\{P_\alpha\}_{\alpha\in\calS}$, we alternately optimize the controller gains $\{K_A\}_{A\in\calS_\calH}$ and the matrices $\{P_\alpha\}_{\alpha\in\calS}$ while minimizing $\gamma$. For fixed $\{P_\alpha\}_{\alpha\in\calS}$, condition~\eqref{eq:stabilization_LMIs} can be reformulated as a set of LMIs using the Schur complement. The controller synthesis problem is then convex, and controller gains minimizing $\gamma$ can be computed. For fixed $\{K_A\}_{A\in\calS_\calH}$, feasibility with respect to $\{P_\alpha\}_{\alpha\in\calS}$ is checked via a semidefinite program, and the smallest admissible value of $\gamma$ is determined by bisection. These two steps are repeated until a solution with $\gamma<1$ is obtained.

The resulting pair $(\{K_A\}_{A\in\calS_\calH},\{P_\alpha\}_{\alpha\in\calS})$ is then used to initialize the alternating optimization procedure described in Algorithm~\ref{alg:alternating_optimization} (cf.~line 2).

\section{SDP formulation of \texorpdfstring{\eqref{eq:K_step}}{(1)}}\label{sec:SDP_formu}
For a fixed value of $\gamma>0$, the constraints of~\eqref{eq:K_step} can be rewritten as
\[
\gamma P_\alpha
-
Q
-
K_A^\top R K_A
-
(A_i+B_iK_A)^\top \gamma P_\beta (A_i+B_iK_A)
\succeq 0.
\]
An application of the Schur complement shows that the above inequality is equivalent to
\[
\begin{bmatrix}
\gamma P_\alpha - Q &
K_A^\top &
(A_i+B_iK_A)^\top
\\[1ex]
K_A &
R^{-1} &
0
\\[1ex]
A_i+B_iK_A &
0 &
\frac1\gamma P_\beta^{-1}
\end{bmatrix}
\succeq 0.
\]
For fixed $\gamma$, $P_\alpha$, and $P_\beta$, this is a linear matrix inequality in the variable $K_A$.

\section{De Bruijn graphs}\label{sec:DB}

In this appendix, we recall the definition of primal and dual De Bruijn graphs used in the numerical experiments.

\begin{definition}[Primal De Bruijn graph]
Let $\ell\in\Ne_{>0}$ and let $\Mset$ be a finite set. The \emph{primal De Bruijn graph} of order $\ell$ is the graph
\[
\calH_\ell(\Mset)=(\calS_\ell,\calE_\ell),
\]
with node set
\[
\calS_\ell=\Mset^\ell.
\]
There is an edge $(\alpha,\beta,i)\in\calE_\ell$ if and only if
\[
\alpha=(j_1,\ldots,j_\ell),
\qquad
\beta=(i,j_1,\ldots,j_{\ell-1}),
\]
for some $(j_1,\ldots,j_\ell)\in\Mset^\ell$.
\end{definition}
\begin{definition}[Dual De Bruijn graph]
Let $\ell\in\Ne_{>0}$ and let $\Mset$ be a finite set. The \emph{dual De Bruijn graph} of order $\ell$ is the graph
\[
\calH_\ell(\Mset)=(\calS_\ell,\calE_\ell),
\]
with node set
\[
\calS_\ell=\Mset^\ell.
\]
There is an edge $(\alpha,\beta,i)\in\calE_\ell$ if and only if
\[
\alpha=(i,j_1,\ldots,j_{\ell-1}), \qquad
\beta=(j_1,\ldots,j_\ell),
\]
for some $(j_1,\ldots,j_\ell)\in\Mset^\ell$.
\end{definition}

Notice that the dual De Bruijn graph of order $\ell$ is obtained by reversing each edge of its primal counterpart. 

Primal De Bruijn graphs of orders $1$ to $3$ are shown in Figure~\ref{fig:db_examples} for the mode set $\Mset=\{1,2\}$.
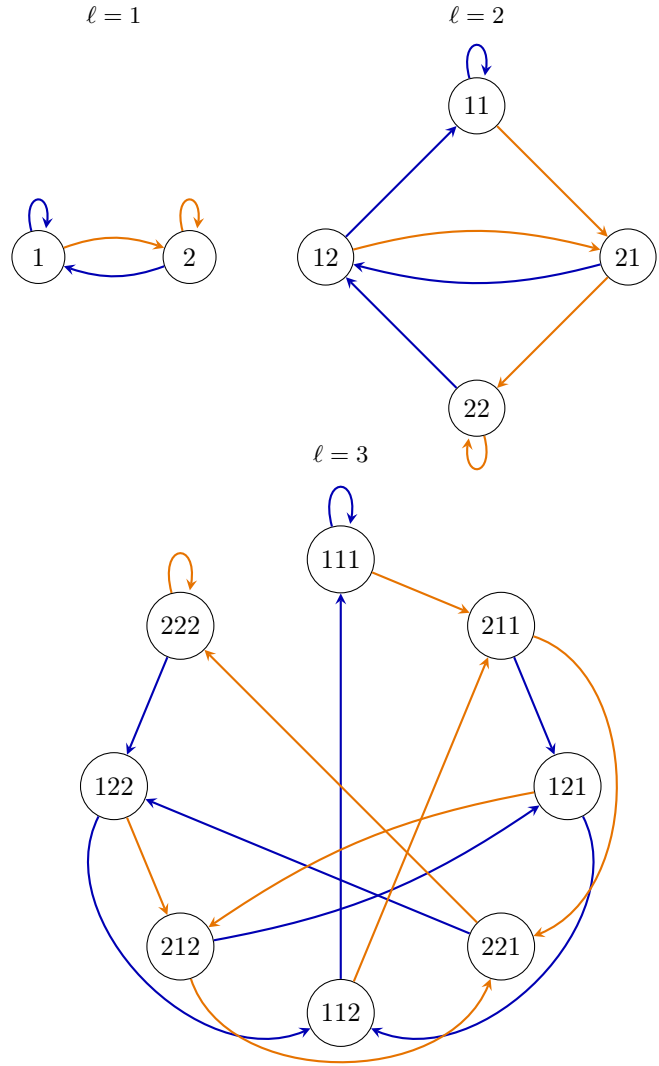
\begin{figure}[t]
\centering

\begin{tikzpicture}[
    >=stealth,
    every node/.style={circle, draw, minimum size=7mm},
e1/.style={->, thick, blue!70!black},
e2/.style={->, thick, orange!90!black},
]


\begin{scope}[shift={(-4,3)}]

\node[draw=none,font=\small] at (0,3.2) {$\ell=1$};

\node (1) at (-1,0) {1};
\node (2) at (1,0)  {2};

\path[e1]
(1) edge[loop above] (1)
(2) edge[bend left=20] (1);

\path[e2]
(1) edge[bend left=20] (2)
(2) edge[loop above] (2);

\end{scope}


\begin{scope}[shift={(0.8,3)}]

\node[draw=none,font=\small] at (0,3.2) {$\ell=2$};

\node (11) at (90:2)  {11};
\node (12) at (0:2)   {21};
\node (22) at (-90:2) {22};
\node (21) at (180:2) {12};

\path[e1]
(11) edge[loop above] (11)
(12) edge[bend left=15] (21)
(21) edge (11)
(22) edge (21);

\path[e2]
(11) edge (12)
(12) edge (22)
(21) edge[bend left=15] (12)
(22) edge[loop below] (22);

\end{scope}


\begin{scope}[shift={(-1,-4)}]

\node[draw=none,font=\small] at (0,4.4) {$\ell=3$};

\node (111) at (90:3)  {111};
\node (112) at (45:3)  {211};
\node (121) at (0:3)   {121};
\node (122) at (-45:3) {221};
\node (211) at (-90:3) {112};
\node (212) at (-135:3){212};
\node (221) at (180:3) {122};
\node (222) at (135:3) {222};

\path[e1]
(111) edge[loop above] (111)
(112) edge (121)
(121) edge[bend left=72] (211)
(122) edge (221)
(211) edge (111)
(212) edge[bend right=12] (121)
(221) edge[bend right=72] (211)
(222) edge (221);

\path[e2]
(111) edge (112)
(112) edge[bend left=72] (122)
(121) edge[bend right=12] (212)
(122) edge (222)
(211) edge (112)
(212) edge[bend right=72] (122)
(221) edge (212)
(222) edge[loop above] (222);

\end{scope}

\end{tikzpicture}
\caption{Primal De Bruijn graphs of orders $\ell=1$, $\ell=2$, and $\ell=3$ for the mode set $\Mset=\{1,2\}$. Blue and orange edges correspond to labels $1$ and $2$, respectively.}
\label{fig:db_examples}

\end{figure}

\end{document}